\documentclass[11pt]{article}
\usepackage[utf8]{inputenc}
\usepackage{amsmath}
\usepackage{amssymb}
\usepackage{amstext}
\usepackage{amsthm}
\usepackage[linesnumbered,ruled,vlined]{algorithm2e}
\usepackage{algorithmic}
\usepackage{fullwidth}
\usepackage{float}
\usepackage{colonequals}
\usepackage{hyperref}
\usepackage{comment}
\usepackage{graphicx}
\usepackage{xcolor}         

\newtheorem{theorem}{Theorem}[section]
\newtheorem{lemma}[theorem]{Lemma}

\newtheorem{proposition}[theorem]{Proposition}
\newtheorem{corollary}[theorem]{Corollary}

\newtheorem{definition}[theorem]{Definition}
\newtheorem{remark}[theorem]{Remark}

\newcommand{\RR}{\mathbb{R}}

\newcommand{\NN}{\mathbb{N}}

\newcommand{\PP}{\mathbb{P}}
\newcommand{\EE}{\mathbb{E}}

\newcommand{\vp}{\varphi}

\newcommand{\sI}{\mathcal{I}}

\newcommand{\sN}{\mathcal{N}}

\newcommand{\one}{{\bf 1}}

\newcommand{\la}{\langle}
\newcommand{\ra}{\rangle}

\newcommand{\Unif}{\mathsf{Unif}}
\newcommand{\Vol}{\mathsf{Vol}}

\title{Asymptotics of smoothed Wasserstein distances \\ in the small noise regime}
\date{}

\usepackage{authblk}
\author[1]{Yunzi Ding}
\author[2]{Jonathan Niles-Weed\thanks{Supported in part by NSF grant DMS-2015291.}}

\affil[1]{Courant Institute of Mathematical Sciences, NYU}
\affil[2]{Courant Institute of Mathematical Sciences and the Center for Data Science, NYU}

\begin{document}

\maketitle

\begin{abstract}
  We study the behavior of the Wasserstein-$2$ distance between discrete measures $\mu$ and $\nu$ in $\RR^d$ when both measures are smoothed by small amounts of Gaussian noise. This procedure, known as \emph{Gaussian-smoothed optimal transport}, has recently attracted attention as a statistically attractive alternative to the unregularized Wasserstein distance. We give precise bounds on the approximation properties of this proposal in the small noise regime, and establish the existence of a phase transition: we show that, if the optimal transport plan from $\mu$ to $\nu$ is unique and a perfect matching, there exists a critical threshold such that the difference between $W_2(\mu, \nu)$ and the Gaussian-smoothed OT distance $W_2(\mu \ast \sN_\sigma, \nu\ast \sN_\sigma)$ scales like $\exp(-c /\sigma^2)$ for $\sigma$ below the threshold, and scales like $\sigma$ above it. These results establish that for $\sigma$ sufficiently small, the smoothed Wasserstein distance approximates the unregularized distance exponentially well.
\end{abstract}

\section{Introduction: optimal transport}
Optimal Transport (OT) has seen a recent surge of applications in machine learning, in areas such as generative modeling~\cite{ArjChiBot17,GenPeyCut18}, image processing~\cite{PitKokDah07, RubTomGui00,FeyChaVia17}, and domain adaptation~\cite{CouFlaTui14, CouFlaTui17}. A natural statistical question raised by these applications is to estimate the OT distances with samples. 
These distances, known as the Wasserstein distances, are defined by
\begin{equation*}
W_p^p(\mu, \nu) = \inf_{\pi \in \Pi(\mu, \nu)} \int \|x - y\|^p d \pi(x, y)\,,
\end{equation*}
where $\Pi(\mu, \nu)$ denotes the set of joint measures with marginals $\mu$ and $\nu$, known as \emph{transport plans}.
It is well known that plug-in estimators for this quantity, obtained by replacing $\mu$ and $\nu$ with empirical measures consisting of i.i.d.\ samples, have performance in high dimensions, with rates of convergence typically of order $n^{-p/d}$~\cite{dudley-convergence,blg-empirical,dobric-asymptotics,FG-rate,ManNil21} when $d > 2p$.
Moreover, minimax lower bounds show that this curse of dimensionality is unavoidable in general~\cite{SP-minimax,NR-stm}.

The existence of the curse of dimensionality for OT has led to a series of proposals to obtain better rates of convergence by imposing additional structural assumptions---such as latent low-dimensionality~\cite{NR-stm} or smoothness~\cite{SinUppLi18,NB-minimax-smooth}---or by replacing $W_p$ by a better-behaved surrogate, such as an entropy-regularized version with much better statistical and computational properties~\cite{cuturi-sinkhorn,genevay-complexity,RW-entropic,MN-entropic,awr-sinkhorn}.

A particularly intriguing option, developed by~\cite{GG-gaussian}, consists in \emph{smoothing} the Wasserstein distance by adding Gaussian noise.
The following result shows the statistical benefits of this approach.

\begin{proposition}[\cite{GGNP-empirical}]
For $d > 1$ and $\sigma > 0$, denote by $\sN_\sigma$ the centered Gaussian measure on $\RR^d$ with covariance $\sigma^2 I_d$. For any compactly supported probability measure $\mu$ in $\RR^d$, let $x_1, x_2, \dots, x_n$ be i.i.d.\ samples from $\mu$, and define the empirical measure
\[\hat{\mu}_n = \frac{1}{n}\sum_{i = 1}^n \delta(x_i).\]
Then there exists a constant $c = c(\mu, \sigma)$ such that
\begin{equation*}
    \EE W_2(\hat{\mu}_n \ast \sN_\sigma, \mu\ast\sN_\sigma) \le cn^{-1/2}.
    \end{equation*}
\end{proposition}
\noindent
\cite{GG-gaussian} call this framework \emph{Gaussian-smoothed optimal transport} (GOT), and follow up work has shown that it possesses significant statistical benefits, with fast rates of convergence and clean limit laws~\cite{ZCR-smoothed,goldfeld-estim-NN,GGK-generative,GolKatNie22}.

To leverage the beneficial properties of the GOT framework, it is necessary to understand how well the smoothed distance $W_2(\mu\ast\sN_\sigma, \nu\ast\sN_\sigma)$ approximates the standard Wasserstein distance $W_2(\mu, \nu)$.
An application of the triangle inequality shows that
\begin{equation}\label{eq:naive_estimate}
|W_2(\mu, \nu) - W_2(\mu\ast\sN_\sigma, \nu\ast\sN_\sigma)| \lesssim \sigma\,.
\end{equation}
Indeed, the triangle inequality implies $|W_2(\mu\ast\sN_\sigma, \nu\ast\sN_\sigma) - W_2(\mu, \nu)| \leq W_2(\mu, \mu \ast\sN_\sigma) + W_2(\nu, \nu \ast \sN_\sigma)$ and the latter two terms are of order at most $\sigma$. In general, this upper bound is unimprovable, as we show below.
On the other hand, it can also be very loose: if $\mu$ is a translation of $\nu$, then $W_2(\mu\ast\sN_\sigma, \nu\ast\sN_\sigma) = W_2(\mu, \nu)$ for all $\sigma \geq 0$.
These examples raise a natural question: how well does $W_2(\mu\ast\sN_\sigma, \nu\ast\sN_\sigma)$ approximate $W_2(\mu, \nu)$ when $\sigma$ is small, and how does the answer to this question depend on the measures $\mu$ and $\nu$?


The main goal if this paper is to give a sharp answer to this question for \emph{finitely supported} measures.
We focus on the finite support case for two reasons.
First, when $\mu$ and $\nu$ are finitely supported, $\mu \ast \sN_\sigma$ and $\nu \ast \sN_\sigma$ are each finite mixtures of Gaussians, and the behavior of Wasserstein distances for such measures is a topic of active research~\cite{doi:10.1137/19M1301047,DBLP:journals/access/ChenGT19}.
Second, as our results indicate, the behavior of this quantity for finitely supported measures is unexpectedly rich, with a sharp dichotomy in rates depending on the structure of the optimal transport plan between $\mu$ and $\nu$:
we show that when the \emph{unique} optimal transport plan between $\mu$ and $\nu$ is a \emph{perfect matching}, then there exist positive $\sigma_*$ and $c$ such that
\begin{equation*}
0 \leq W_2(\mu, \nu) - W_2(\mu\ast\sN_\sigma, \nu\ast\sN_\sigma) \lesssim e^{-c/\sigma^2} \quad \forall \sigma \in (0, \sigma_*)\,.
\end{equation*}
In other words, for sufficiently small $\sigma$, the GOT distance approximates the standard $W_2$ distance exponentially well, substantially sharpening~\eqref{eq:naive_estimate}.
More strikingly, we establish the existence of a phase transition: for $\sigma < \sigma_*$, the gap is exponentially small, whereas for $\sigma > \sigma_*$, the gap scales linearly.
By contrast, if the optimal transport plan between $\mu$ and $\nu$ is not unique or is not a perfect matching, then no phase transition appears: the upper bound of \eqref{eq:naive_estimate} is tight even in a neighborhood of $\sigma = 0$.

To locate exactly where the phase transition happens, we introduce a notion of robustness of the optimal transport plan between $\mu$ and $\nu$, which is motivated by the concept of \textit{cyclical monotonicity}~\cite{rocka1,rocka2,rochet}. 
(See definition in Section~\ref{sec:main}.)
A fundamental result in the theory of optimal transport~\cite{rocka1} is that the support of the optimal transport plan in the definition of $W_2^2$ is cyclically monotone.
We define a robust version of this property and show that it characterizes measures for which the gap between $W_2(\mu, \nu)$ and $W_2(\mu\ast\sN_\sigma, \nu\ast\sN_\sigma)$ is exponentially small.
We show that the critical $\sigma_*$ can be described in terms of the strong convexity of the \emph{potentials} appearing in the dual of the optimal transport problem.
The strong convexity of these potentials has previously been explored in computational and statistical contexts~\cite{vacher2021convex,paty2020regularity}, but to our knowledge its connection to Gaussian smoothed optimal transport is new.

Our work provides a precise understanding on how GOT resembles vanilla OT in the vanishing noise ($\sigma\downarrow 0$) regime.
These results complement those recently obtained by \cite{CN-asymp-large-noise} in the large noise regime, who show that if $\mu$ and $\nu$ have $n$ matching moments, $n \geq 1$, then $W_2(\mu\ast\sN_\sigma, \nu\ast\sN_\sigma) = O(\sigma^{-n})$ as $\sigma\rightarrow\infty$.

Along with results in~\cite{CN-asymp-large-noise}, our work completes the limiting picture of the Euclidean heat semigroup acting on atomic measures under the Wasserstein distance. All the relevant rates are presented in Table~\ref{tab:1}.

We note that our work leaves open the question of characterizing the rates for non-atomic measures.
It is possible to show that, for general measures, there are measures exhibiting polynomial rates intermediate between $\sigma$ and $e^{-c/\sigma^2}$; however, these rates appear to depend delicately on the geometry of the measures and their support.
Giving a full characterization of the rate for general probability measures is an attractive open question.
\begin{table*}[t]
  \centering
\begin{tabular}{ c | c | c | c | c} 
 \hline
 Regime & Condition & $\lim (W_2(\mu\ast\sN_\sigma, \nu\ast\sN_\sigma))$ & Rate & Reference \\
 \hline\hline
 $\sigma \downarrow 0$ & Unique OT plan & $W_2(\mu, \nu)$ & $e^{-c/\sigma^2}$ & Theorem~\ref{thm:main-exp} \\ 
 \hline
 $\sigma \downarrow 0$ & Non-unique OT plan & $W_2(\mu, \nu)$ & $\sigma$ & Theorem~\ref{thm:main-linear} \\ 
 \hline
 $\sigma\uparrow \infty$ & $\mu$ and $\nu$ agree up to $n$th moment & $0$ & $\sigma^{-n}$ & \cite{CN-asymp-large-noise}\\
 \hline
\end{tabular}
  \caption{Limiting behavior of $W_2(\mu\ast\sN_\sigma, \nu\ast\sN_\sigma)$ for atomic measures $\mu$ and $\nu$.}
  \label{tab:1}
\end{table*}

\section{Cyclical monotonicity and implementability}
\label{sec:main}
We are concerned with the optimal transport problem between discrete measures
\[\mu = \sum_{i = 1}^k \alpha_i \delta(x_i), \quad \nu = \sum_{j = 1}^k \alpha_j \delta(y_j)\]
in the space $\RR^d$, equipped with the squared Euclidean cost function $c(x, y) = \|x-y\|^2/2$. (The generalization to discrete measures with different numbers of atoms and weights is considered in Section~\ref{sec:not-perfect-matching}.) We are mainly interested in transport plans in the form of perfect matchings between $\{x_i\}$ and $\{y_i\}$. By relabeling the points, we may assume without loss of generality that the optimal transport plan between $\mu$ and $\nu$ is the unique coupling with support
\[\Gamma = \left\{(x_1, y_1), (x_2, y_2), \dots, (x_k ,y_k)\right\}.\]

Our techniques are based on a robust notion of optimality for $\Gamma$.
We recall the following definition of cyclical monotonicity, which serves as an important certification of an optimal transport plan.
\begin{definition}[See, e.g.,~\cite{rocka1}]
\label{def:cyclical-monotone}
A set $S \subseteq \RR^d \times \RR^d$ is \emph{cyclically monotone} if for any $(a_1, b_1), \dots, (a_n, b_n) \in S$, we have
\begin{equation*}
\sum_{i=1}^n \|a_i - b_i\|^2 \leq \sum_{i=1}^n \|a_i - b_{i+1}\|^2\,,
\end{equation*}
where we set $b_{n+1} := b_1$.
\end{definition}
The significance of this notion is the following fundamental result.
\begin{theorem}[{See~\cite[Theorem 5.10]{Vil08}}]\label{monotone_opt}
If $\pi \in \Pi(\mu, \nu)$ has cyclically monotone support, then it is an optimal transport plan between $\mu$ and $\nu$.
\end{theorem}

We strengthen this notion by insisting that the inequalities in the definition of cyclical monotonicity be strict.
\begin{definition}
We say $f: [k]\times [k]\rightarrow \RR_{\ge 0}$ is a positive residual function on $[k]$, if $f(i, i) = 0$, $f(i, j) > 0$ for $i \neq j$, and $f(i, j) = f(j, i)$ for all $i, j\in [k]$.
\end{definition}

\begin{definition}[Strong cyclical monotonicity]
For a positive residual function $f$ on $[k]$, we say that $\Gamma$ is $f$-strongly cyclically monotone, if for any $n\in \NN$ and $\sigma(1), \sigma(2), \dots, \sigma(n) \in [k]$ with $\sigma(i)\neq \sigma(i+1)$ (the convention is $\sigma(n+1) = \sigma(1)$), we have
\begin{equation*}
 \sum_{i = 1}^n \|x_{\sigma(i)} - y_{\sigma(i)}\|^2 
 \le \sum_{i = 1}^n \|x_{\sigma(i)} - y_{\sigma(i+1)}\|^2 - \sum_{i = 1}^n f(\sigma(i), \sigma(i+1)),    
\end{equation*}
or equivalently,
\begin{equation*}
    \sum_{i = 1}^n \la x_{\sigma(i)}, y_{\sigma(i)}-y_{\sigma(i+1)} \ra \ge  \sum_{i = 1}^n f(\sigma(i), \sigma(i+1)).
\end{equation*}
\end{definition}
\noindent
Strong cyclical monotonicity indicates that the optimal plan with support $\Gamma$ is superior to any other plan by a positive margin in its transport cost. In~\cite{rochet}, the author introduced the notion \textit{implementability} and established it as an equivalent condition of cyclical monotonicity. In parallel to the results in~\cite{rochet}, we also consider the following stronger condition of implementability.
\begin{definition}[Strong implementability]
For a positive residual function $f$ on $[k]$, we say that $\Gamma$ is $f$-strongly implementable, if there exists a potential function $\vp$, such that for any $i, j\in [k]$, we have
\begin{equation*}
    \la x_i, y_i - y_j \ra \ge \vp(y_i) - \vp(y_j) + f(i, j).
\end{equation*}
\end{definition}
\noindent
Analogous to the equivalence result in~\cite{rochet}, we show that strong cyclical monotonicity and strong implementability are both equivalent to the uniqueness and optimality of $\Gamma$.
\begin{proposition}\label{prop:equiv}
The following three statements are equivalent:
\begin{itemize}
    \item[(i)] $\Gamma$ is $f$-strongly cyclically monotone for some $f$;
    \item[(ii)] $\Gamma$ is $f$-strongly implementable;
    \item[(iii)] $\Gamma$ is the unique optimal transport plan from $\{x_i\}$ to $\{y_i\}$.
\end{itemize}
\end{proposition}
\noindent
The positive payment function constructed in the equivalence between (iii) and (i) in Proposition~\ref{prop:equiv} is of the form
$f(i, j) = \frac{\lambda}{2}\|y_i - y_j\|^2$ for some $\lambda > 0$, in which case the implementability condition reads
\[\la x_i, y_i - y_j\ra \ge \vp(y_i) - \vp(y_j) + \frac{\lambda}{2} \|y_i - y_j\|^2.\]
This condition is equivalent to the existence of a $\lambda$-strongly convex potential $\vp$ satisfying $\nabla \vp(y_i) = x_i$ for all $i \in [k]$~\cite{TayHenGli17}, or, equivalently, the existence of a Lipschitz \emph{Brenier map} from $\mu$ to $\nu$~\cite{Bre87}.
More generally, we have the following theorem characterizing the properties of strongly implementable plans with residual functions of quadratic type.

\begin{theorem}
\label{thm:convex-equiv}
The following conditions are equivalent:
\begin{itemize}
    \item [(i)] For some positive numbers $\alpha < \beta$, there exists a potential function $\vp:\RR^d \rightarrow \RR^d$ which is $\alpha$-strongly convex and $\beta$-smooth, such that $x_i = \nabla \vp(y_i)$ for all $i\in [k]$.
    \item [(ii)] $\Gamma$ is strongly implementable for
    \begin{equation}
    \label{cond:str-imp}
        f(i, j) := \frac{1}{2(\beta - \alpha)}\left(\|x_i - x_j\|^2
         + \alpha\beta \|y_i - y_j\|^2 - 2\alpha \la y_i - y_j, x_i - x_j\ra\right),
    \end{equation}
    or equivalently, there exists $\{\tilde{\vp}(y_i)\}_{i = 1}^k \subset \RR^d$, such that for all $i, j\in [k]$ ($i\neq j$),
    \begin{equation}
    \label{cond:str-imp-2}
    \begin{aligned}
        \la x_i, y_i - y_j\ra \ge &\ \tilde{\vp}(y_i) - \tilde{\vp}(y_j) \\
        &+ \frac{1}{2(\beta - \alpha)}\left(\|x_i - x_j\|^2
         + \alpha\beta \|y_i - y_j\|^2 - 2\alpha \la y_i - y_j, x_i - x_j\ra\right)
         \end{aligned}
    \end{equation}
\end{itemize}
\end{theorem}

\begin{proof}
This is a direct application of Theorem 4 in~\cite{TayHenGli17}.
\end{proof}
\begin{remark}
We should emphasize that the $f$ defined in Theorem~\ref{thm:convex-equiv} is indeed a positive residual function given $\alpha < \beta$, since Cauchy-Schwartz gives
\[2\alpha \la y_i - y_j, x_i - x_j \ra 
\le \|x_i - x_j\|^2 + \alpha^2 \|y_i - y_j\|^2 
< \|x_i - x_j\|^2 + \alpha\beta \|y_i - y_j\|^2.\]
\end{remark}

\noindent
As a direct consequence of the direction (ii) to (i) in Theorem~\ref{thm:convex-equiv}, if $\Gamma$ is strongly implementable for a positive residual function $f$ which is quadratic in $y_i - y_j$ and $x_i - x_j$, we will have guarantee on strong convexity and smoothness of the potential function.
\begin{corollary}
Suppose $\Gamma$ is strongly implementable for 
\begin{equation*}
    f(i, j) = \frac{1}{2}\left(\lambda_{xx}\|x_i - x_j\|^2
         + \lambda_{yy} \|y_i - y_j\|^2 - 2\lambda_{xy} \la y_i - y_j, x_i - x_j\ra\right)
\end{equation*}
where $\lambda_{xx}, \lambda_{xy}$ and $\lambda_{yy}$ are positive numbers which satisfy $\lambda_{xy}^2 + \lambda_{xy} = \lambda_{xx}\lambda_{yy}$. Then there exists a potential function $\vp: \RR^d \rightarrow \RR^d$ which is $\frac{\lambda_{xy}}{\lambda_{xx}}$-strongly convex and $\frac{\lambda_{yy}}{\lambda_{xy}}$-smooth, such that $\nabla x_i = \vp(y_i)$ for all $i\in [k]$. 
\end{corollary}

\noindent
A crucial property of strongly cyclical monotone (or strongly implementable) transport plans is that they are robust to small perturbations in the sources and targets. We quantify the robustness of the map $\Gamma$ in the following definition.
\begin{definition}
\label{def:eps-robust}
For $\epsilon \ge 0$, we say $\Gamma$ is $\epsilon$-robust, if for any distinct $\sigma(1), \sigma(2), \dots, \sigma(n) \in [k]$, and any $\alpha_{\sigma(1)}, \alpha_{\sigma(2)}, \dots, \alpha_{\sigma(n)}\in \RR^d$ such that
\[\max_{i} \|\alpha_{\sigma(i)}\| \le \epsilon,\]
there holds
\begin{equation*}
     \sum_{i = 1}^n \|x_{\sigma(i)} - y_{\sigma(i)}\|^2 
     \le \sum_{i = 1}^n \|(x_{\sigma(i)}+\alpha_{\sigma(i)})
     - (y_{\sigma(i+1)} + \alpha_{\sigma(i+1)})\|^2.
\end{equation*}
In the case that $\Gamma$ is an optimal transport plan, also denote
\begin{equation*}
    R(\Gamma) := \sup \left\{\epsilon \ge 0\ :\ \Gamma\ \text{is}\ \epsilon\text{-robust}\right\}.
\end{equation*}
\end{definition}
\noindent
The quantity $R(\Gamma)$, which we call ``robustness of optimality", is crucial to understanding the behavior of the optimal transport cost between $\mu$ and $\nu$ corrupted with noise.
\begin{proposition}\label{prop:equiv-scm}
$\Gamma$ is strongly cyclically monotone if and only if $R(\Gamma)>0$.
\end{proposition}
\noindent
The following proposition quantifies the relation between $R(\Gamma)$ and a positive residual $f$ for which $\Gamma$ is strongly implementable, and provides a certification of a lower bound of $R(\Gamma)$ in $O(k^2)$ time.

\begin{proposition}
\label{prop:lb-eps-1}
Suppose $\Gamma$ is strongly implementable for a positive residual function $f$. Then $T$ is $\epsilon$-robust for 
\begin{equation}
\label{eq:eps-up-bound}
    \epsilon \le \frac{1}{2}\inf_{i\neq j} \frac{f(i, j)}{\|x_i - x_j\| + \|y_i - y_j\|}.
\end{equation}
This implies that
\begin{equation*}
    R(\Gamma) \ge \frac{1}{2}\inf_{i\neq j} \frac{f(i, j)}{\|x_i - x_j\| + \|y_i - y_j\|}.
\end{equation*}
\end{proposition}
A special case of Proposition~\ref{prop:lb-eps-1} is when the optimal transport plan is strongly implementable with residual functions of quadratic type. In this case, we are able to derive a simple closed-form lower bound of $R(\Gamma)$.
\begin{proposition}
\label{prop:lb-R}
When the equivalence in Theorem~\ref{thm:convex-equiv} holds, $T$ is $\epsilon$-robust for 
\begin{equation}
\label{cond:eps-robust-cond}
    \epsilon \le \frac{1}{2}\inf_{i\neq j} \frac{\max\left\{\frac{1}{\beta}\|x_i - x_j\|^2 , \alpha \|y_i - y_j\|^2\right\}}{\|x_i - x_j\| + \|y_i - y_j\|}.
\end{equation}
This implies that 
\begin{equation*}
     R(\Gamma)
    \ge \frac{1}{2}\inf_{i\neq j} \frac{\max\left\{\frac{1}{\beta}\|x_i - x_j\|^2 , \alpha \|y_i - y_j\|^2\right\}}{\|x_i - x_j\| + \|y_i - y_j\|}.
\end{equation*}
\end{proposition}
\begin{remark}
When condition (i) in Theorem~\ref{thm:convex-equiv} holds, $\alpha$-strong convexity and $\beta$-smoothness implies
\[\frac{1}{\beta}\|x_i - x_j\| \le \|y_i - y_j\| \le  \frac{1}{\alpha}\|x_i - x_j\|.\]
Thus the condition~\eqref{cond:eps-robust-cond} may be replaced by the bound
\begin{equation}
\begin{aligned}
    \label{cond:eps-robust-cond-strong}
    \epsilon \le &\ \frac{1}{2}\inf_{i\neq j}\max\left\{\frac{\alpha}{1+\beta}\|x_i - x_j\|, \frac{\alpha}{\beta(1+\alpha)}\|y_i - y_j\|\right\},
    \end{aligned}
\end{equation}
which is easier to verify in practice.
\end{remark}
\noindent


\section{Case I: perfect matching}
Our main results show that the robustness of optimality $R(\Gamma)$ controls the gap between $W_2(\mu, \nu)$ and $W_2(\mu\ast\sN_\sigma, \nu\ast\sN_\sigma)$.
\begin{theorem}
\label{thm:main-exp}
If $\sigma_\ast = R(\Gamma)>0$, then for $\sigma\in (0, \sigma_\ast)$,
\begin{equation*}
    W_2(\mu, \nu) - W_2(\mu\ast\sN_\sigma, \nu\ast\sN_\sigma) \lesssim \sqrt{\sigma_\ast \sigma}e^{-\sigma_\ast^2 / 4\sigma^2}.
\end{equation*}
\end{theorem}

\noindent
 In the regime where $\sigma$ does not exceed $R(\Gamma)$, the above theorem tells that the GOT distance is an excellent approximation of the OT distance.
 Our second main result is a converse to that statement, showing that if $\sigma$ goes beyond $R(\Gamma)$, we show that the loss $W_2(\mu, \nu) - W_2(\mu\ast \sN_\sigma, \nu\ast\sN_\sigma)$ is linear in $\sigma$.
 We start with the following proposition, which quantifies a ``violation of cyclical monotonicity" under possibly large perturbations in the sources and targets.
\begin{proposition}
\label{prop:max-local-improve}
If $\Gamma$ is an optimal transport plan, for any $M \ge 0$, denote
\begin{equation*}
    G(M) := \sup\left\{ \sum_{i = 1}^n \|x_{\sigma(i)} - y_{\sigma(i)}\|^2
      - \sum_{i = 1}^n \|(x_{\sigma(i)}+\alpha_{\sigma(i)}) - (y_{\sigma(i+1)}+\alpha_{\sigma(i+1)})\|^2\ :
     \|\alpha_{\sigma(i)}\| \le M\right\}
\end{equation*}
Then $G(M)$ is a concave function of $M$ for $M\in [0, +\infty)$.
\end{proposition}
\noindent

Note that $G(M)$ vanishes for $M < \sigma_*$. The next theorem shows that as long as $G(M)$ is not negligible for $M \gtrsim \sigma_*$, the approximation loss for $\sigma \geq \sigma_*$ is linear in $\sigma$.

\begin{theorem}
\label{thm:main-linear}
If $\sigma_\ast = R(\Gamma)>0$, then
\begin{equation*}
    W_2^2(\mu, \nu) - W_2^2(\mu\ast \sN_\sigma, \nu\ast \sN_\sigma)
    \gtrsim \sup_{M > \sigma_\ast} e^{-cM^2/\sigma^2} G(M).
\end{equation*}
Here $G(M)$ is defined as in Proposition~\ref{prop:max-local-improve}. In particular, if $G(3 \sigma_*) \gtrsim \sigma_*$, then for $\sigma\in (0, 2\sigma_\ast)$,
\begin{equation*}
    W_2^2(\mu, \nu) - W_2^2(\mu\ast \sN_\sigma, \nu\ast \sN_\sigma)
    \gtrsim \sigma e^{-c\sigma_\ast^2/\sigma^2}.
\end{equation*}
\end{theorem}


To prove Theorem~\ref{thm:main-exp}, we need the following lemma, which tells that no loss in $W_2$ is incurred by a local perturbation on $\mu$ and $\nu$. The proof of Lemma~\ref{lemma:local-no-change} can be found in Section~\ref{sec:aux-proofs}.
\begin{lemma}
\label{lemma:local-no-change}
If $\sigma_\ast = R(\Gamma)>0$, then for any measure $Q$ in $\RR^d$ supported on $B(0, \sigma_\ast)$,
\[W_2(\mu, \nu)  = W_2(\mu\ast Q, \nu\ast Q).\]
\end{lemma}

\begin{proof}[Proof of Theorem~\ref{thm:main-linear}]
For $M>\sigma_\ast$, pick $\sigma(1), \sigma(2), \dots, \sigma(n) \in [k]$ and $\{\alpha_{\sigma(i)}\}_{i = 1}^n \subset \RR^d$ such that $\|\alpha_{\sigma(i)}\|\le M$ and
\begin{align*}
    \sum_{i = 1}^n \|x_{\sigma(i)} - y_{\sigma(i)}\|^2  - \sum_{i = 1}^n \|(x_{\sigma(i)}+\alpha_{\sigma(i)}) - (y_{\sigma(i+1)}
 + \alpha_{\sigma(i+1)})\|^2 = G(M).
\end{align*}
 For every $i\in [k]$, denote $B_{\sigma(i)}$ the ball centered at $x_{\sigma(i)}+\alpha_{\sigma(i)}$ with radius $\sigma$, and $\hat{B}_{\sigma(i)}$ the ball centered at $y_{\sigma(i)}+\alpha_{\sigma(i)}$ with radius $\sigma$. Also denote
 \begin{itemize}
     \item $\gamma\in \Pi(\mu\ast\sN_\sigma, \nu\ast\sN_\sigma)$ the law of $(X + Z, Y+ Z)$, where $(X, Y) \sim \frac 1k \sum_{i=1}^k \delta(x_i, y_i)$ and $Z \sim \sN_\sigma$ are independent.
     \item $\gamma_{\sigma(i)}\in \Pi (\Unif(B_{\sigma(i)}), \Unif(\hat{B}_{\sigma(i)}))$ the coupling associated with the transport map
     \[x \mapsto x+y_{\sigma(i)} - x_{\sigma(i)};\]
     \item $\tilde{\gamma}_{\sigma(i)}\in \Pi (\Unif(B_{\sigma(i)}), \Unif(\hat{B}_{\sigma(i+1)}))$ the coupling associated with the transport map
     \[x\mapsto  x+y_{\sigma(i+1)} - x_{\sigma(i)};\]
     \item A constant $m = c_d \exp\left(-\frac{(M+\sigma)^2}{2\sigma^2}\right)$, where $c_d$ is a constant only dependent on the dimension $d$.
 \end{itemize}
 Consider the following measure in $\RR^d\times \RR^d$:
 \begin{equation*}
     \tilde{\gamma} := \gamma - m\sum_{i = 1}^n \gamma_{\sigma(i)} + m\sum_{i = 1}^n \tilde{\gamma}_{\sigma(i)}.
 \end{equation*}
 We shall show that $\tilde{\gamma}\in \Pi(\mu\ast\sN_\sigma, \nu\ast\sN_\sigma)$. We first verify that $\tilde{\gamma}$ is a positive measure on $\RR^d\times \RR^d$. In fact, for $x, y\in \RR^d$, 
\begin{equation*}
    \gamma(dx,dy) = \frac{1}{k}\sum_{i = 1}^k \left(\frac{1}{(\sqrt{2\pi}\sigma)^d} e^{-\frac{\|x-x_i\|^2}{2\sigma^2}} dx
    \cdot \delta_{x-x_i+y_i}(dy)\right).
\end{equation*}
 Meanwhile,
 \begin{equation*}
     \left(m\sum_{i = 1}^n \gamma_{\sigma(i)}\right) (dx, dy) = m\sum_{i = 1}^n \left(\frac{\one \{x\in B_{\sigma(i)}\}}{\Vol(B_{\sigma(i)})} dx \cdot \delta_{x-x_{\sigma(i)} + y_{\sigma(i)}}(dy) \right).
 \end{equation*}
 For every $\sigma(i)$ such that $x\in B_{\sigma(i)}$, note that 
 $$\|x-x_{\sigma(i)}\| \le \|x-(x_{\sigma(i)}+\alpha_{\sigma(i)})\| + \|\alpha_{\sigma(i)}\| \le \sigma+M,$$
 hence (with a proper choice of $c_d$)
 \begin{equation*}
      \frac{1}{k}\frac{1}{(\sqrt{2\pi}\sigma)^d} e^{-\frac{\|x-x_{\sigma(i)}\|^2}{2\sigma^2}}
     \ge  \frac{1}{k}\frac{1}{(\sqrt{2\pi}\sigma)^d} e^{-\frac{(M+\sigma)^2}{2\sigma^2}} 
     \ge \frac{m}{\Vol(B_{\sigma(i)})}.
 \end{equation*}
 As a result, $\gamma - m\sum_{i = 1}^n\gamma_{\sigma(i)} \ge 0$, and $\tilde{\gamma}$ is a positive measure. Also note that its first marginal (i.e. the marginal on the first $d$ dimensions) and second marginal (i.e. the marginal on the last $d$ dimensions) agree with the respective marginals of $\gamma$. Thus we conclude that $\tilde{\gamma}\in \Pi(\mu\ast\sN_\sigma, \nu\ast\sN_\sigma)$. Now note that
 \begin{equation*}
 \begin{aligned}
    & \int c(x, y) d\gamma(x, y) - \int c(x, y)d\tilde{\gamma}(x, y)\\
      =&\ m\left(\sum_{i = 1}^n\|x_{\sigma(i)} - y_{\sigma(i)}\|^2 \right. 
      \left. - \sum_{i = 1}^n \|(x_{\sigma(i)}+\alpha_{\sigma(i)}) - (y_{\sigma(i+1)} +\alpha_{\sigma(i+1)})\|^2\right)\\
      =&\ m\cdot G(M).
      \end{aligned}
 \end{equation*}
 In the meantime, 
 \[\int c(x, y)d\gamma(x, y) = \frac{1}{2k}\sum_{i = 1}^k \|x_i - y_i\|^2 = W_2^2(\mu, \nu),\]
 therefore, 
 \begin{equation*}
 \begin{aligned}
     & W_2^2(\mu\ast \sN_\sigma, \nu\ast\sN_\sigma) \\
     \le & \int c(x, y) d\tilde{\gamma}(x, y) \\
     \le &\ W_2^2(\mu, \nu) - G(M)\cdot c_d\exp\left(-\frac{(M+\sigma)^2}{2\sigma^2}\right).
 \end{aligned}
 \end{equation*}
 In particular, choosing $M = \sigma+\sigma_\ast$ yields
 \[W_2^2(\mu, \nu) - W_2^2(\mu\ast\sN_\sigma, \nu\ast\sN_\sigma) \gtrsim G(\sigma+\sigma_\ast) \exp\left(-c\frac{\sigma_\ast^2}{\sigma^2}\right).\]
 The rest follows from the observation that, for $\sigma \in (0, 2\sigma_\ast)$, 
 \[G(\sigma+\sigma_\ast) = G(\sigma+\sigma_\ast) - G(\sigma_\ast) \ge \frac{G(3\sigma_\ast) - G(\sigma_\ast)}{2\sigma_\ast}\cdot \sigma \]
 since $G$ is concave by Proposition~\ref{prop:max-local-improve}.
\end{proof}


\section{Case II: no perfect matching}
\label{sec:not-perfect-matching}
In the case that $R(\Gamma) = 0$, or equivalently by Proposition~\ref{prop:equiv} and  Proposition~\ref{prop:equiv-scm} that the optimal transport map between $\mu$ and $\nu$ is not a perfect matching, Theorems~\ref{thm:main-exp} and~\ref{thm:main-linear} are not applicable. 
In this situation, we are able to show that the approximation error is linear, even in a neighborhood of zero.
In fact, this holds whenever there exists an optimal transport plan between $\mu$ and $\nu$ which is not a perfect matching.

To analyze this case, we generalize our setting to optimal transport problems between two discrete measures that do not necessarily have the same number of atoms, and whose mass may not be evenly distributed:
\begin{equation}
\label{def:general-discrete}
    \mu = \sum_{i = 1}^m \alpha_i \delta(x_i), \quad \nu = \sum_{j = 1}^n \beta_j\delta(y_j)\,.
\end{equation}
Here $\{\alpha_i\}_{i = 1}^m$ and $\{\beta_j\}_{j = 1}^n$ are positive numbers such that $\sum_{i = 1}^m \alpha_i = \sum_{j = 1}^n \beta_j = 1$. For the sake of notational convenience and without loss of generality, we also assume that $\{x_i\}$ and $\{y_j\}$ are all different. We prove the following result.
\begin{theorem}
\label{thm:no-perfect-matching}
For $\mu, \nu$ defined per~\eqref{def:general-discrete}, unless the optimal transport plan $T$ between $\mu$ and $\nu$ is unique and a perfect matching, i.e. $m = n$ and there exists a permutation $p$ on $[m]$ such that $\alpha_i = \beta_{p(i)}$ and $T^{-1}(y_{p(i)}) = \{x_i\}$ for all $i\in [m]$, there exists $c_0 > 0$ such that for $\sigma\in (0, c_0)$, 
\begin{equation*}
    W_2^2(\mu, \nu) - W_2^2(\mu\ast\sN_\sigma, \nu\ast\sN_\sigma) \gtrsim \sigma.
\end{equation*}
\end{theorem}
Theorem~\ref{thm:no-perfect-matching} tells that, unless the optimal transport plan between $\mu$ and $\nu$ is unique and a perfect matching, the loss from approximating the OT distance with the GOT distance is at least linear in $\sigma$. To proceed with the proof, we need the following lemma. Its proof can be found in Section~\ref{sec:aux-proofs}.
\begin{lemma}
\label{lemma:linear-loss-split}
Let $x, y_1$ and $y_2$ be different points in $\RR^d$. For $\mu_0 := \delta(x)$ and $\nu_0 := \frac{1}{2}\delta(y_1) + \frac{1}{2}\delta(y_2)$, there exists $c_0 > 0$, such that for $\sigma \in (0, c_0)$, we have
\begin{equation}
    W_2^2(\mu_0, \nu_0) - W_2^2(\mu_0\ast\sN_\sigma, \nu_0\ast\sN_\sigma) \gtrsim \sigma.
\end{equation}
\end{lemma}

\begin{proof}[Proof of Theorem~\ref{thm:no-perfect-matching}]
Suppose that there exists a transport plan $\pi$ between $\mu$ and $\nu$ which achieves the optimal cost and is not a perfect matching.
Without loss of generality we assume that $(x_1, y_1), (x_1, y_2) \in \mathrm{supp}(\pi)$.
Let $\lambda = \min\{\pi(x_1, y_1), \pi(x_1, y_2)\}$. We decompose $\mu$ and $\nu$ as
\begin{equation*}
\begin{aligned}
  \hat{\mu} = \mu - 2\lambda\delta(x_1),& \quad \tilde{\mu} = 2\lambda \delta(x_1), \\  
    \hat{\nu} = \nu - \lambda\left(\delta(y_1) + \delta(y_2)\right), &\quad \tilde{\nu} = \lambda\left(\delta(y_1) + \delta(y_2)\right).
\end{aligned}
 \end{equation*}
By Lemma~\ref{lemma:linear-loss-split}, there exists $c_0 >0$ such that for $\sigma\in (0, c_0)$,
\[W_2^2(\tilde{\mu}, \tilde{\nu}) - W_2^2(\tilde{\mu}\ast\sN_\sigma, \tilde{\nu}\ast\sN_\sigma) \gtrsim \sigma.\]
Therefore, for $\sigma \in (0, c_0)$, we also have
\begin{equation*}
    \begin{aligned}
     &W_2^2(\mu, \nu) - W_2^2(\mu\ast\sN_\sigma, \nu\ast\sN_\sigma)\\
     \ge\ & W_2^2(\hat{\mu}, \hat{\nu}) - W_2^2(\hat{\mu}\ast\sN_\sigma, \hat{\nu}\ast\sN_\sigma)
      + W_2^2(\tilde{\mu}, \tilde{\nu}) - W_2^2(\tilde{\mu}\ast\sN_\sigma, \tilde{\nu}\ast\sN_\sigma)\\
       \ge\ & W_2^2(\tilde{\mu}, \tilde{\nu}) - W_2^2(\tilde{\mu}\ast\sN_\sigma, \tilde{\nu}\ast\sN_\sigma)\\
      \gtrsim\ & \sigma,
    \end{aligned}
\end{equation*}
where the first inequality uses that $W_2^2(\mu, \nu) = W_2^2(\hat \mu, \hat \nu) + W_2^2(\tilde \mu, \tilde \nu)$ by the optimality of $\pi$.
\end{proof}
We conclude that, for general discrete measures defined per~\eqref{def:general-discrete}, a phase transition in $W_2(\mu, \nu) - W_2(\mu\ast\sN_\sigma, \nu\ast\sN_\sigma)$ only happens when the optimal transport plan between $\mu$ and $\nu$ is unique and a perfect matching with a positive $R(\Gamma)$. Otherwise, one would always suffer a linear loss in approximating the OT distance with the GOT distance.


\section{Numerical example}
In this section, we present a numerical example to demonstrate different regimes of the rate $W_2(\mu, \nu) - W_2(\mu\ast\sN_\sigma, \nu\ast\sN_\sigma)$, in respect of Theorem~\ref{thm:main-exp} and Theorem~\ref{thm:main-linear}. For the sake of clarity, we consider atomic measures $\mu$ and $\nu$ both defined on $\RR^2$. One of the simplest cases where a coupling $\Gamma$ has $R(\Gamma) = 0$ is
\begin{align*}
   \mu & = \frac{1}{2}\left[\delta((-1, -1)) + \delta((1, 1))\right] ,\\
 \nu &= \frac{1}{2}\left[\delta((-1, 1)) + \delta((1, -1))\right] 
\end{align*}
It is easy to see that the optimal transport plan from $\mu$ to $\nu$ is not unique, which is also a consequence of Proposition~\ref{prop:equiv},  Proposition~\ref{prop:equiv-scm} and the fact that $R(\Gamma) = 0$ for the map
\[\Gamma = \left\{((-1,-1), (-1,1)), ((1,1), (1,-1))\right\}\]
that achieves the optimal cost. We also consider the family 
\[\mu_k = \frac{1}{2}\left[\delta((-1, -1+\frac{k}{10})) + \delta((1, 1-\frac{k}{10}))\right],\  k\in [4]\]
The source and target distributions are demonstrated in Figure 2. For each $k$, the unique optimal transport plan from $\mu_k$ to $\nu$ is given by
\[\Gamma_k = \left\{((-1, -1+\frac{k}{10}), (-1, 1)), ((1, 1-\frac{k}{10}), (1, -1))\right\}.\]
For each of these GOT tasks, we draw $500$ samples from the source distribution $\mu_k\ast\sN_\sigma$ and target distribution $\nu\ast\sN_\sigma$, and use the empirical $W_2$ distance as an estimate of the true $W_2(\mu_k\ast\sN_\sigma, \nu\ast\sN_\sigma)$. We repeat the process $20$ times and report the mean, as shown in the following figure. \\

\graphicspath{ {./} }

\begin{figure}[!htb]
    \centering
    \begin{minipage}{0.5\textwidth}
        \centering
        \includegraphics[width=0.95\linewidth, height=0.17\textheight]{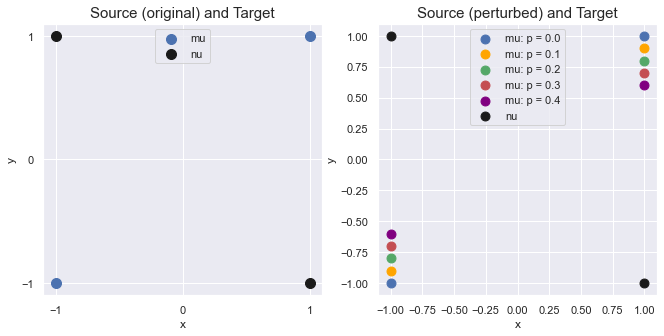}
        \caption{Source and Target distributions}
        \label{fig:distr}
    \end{minipage}%
    \begin{minipage}{0.5\textwidth}
        \centering
        \includegraphics[width=0.95\linewidth, height=0.15\textheight]{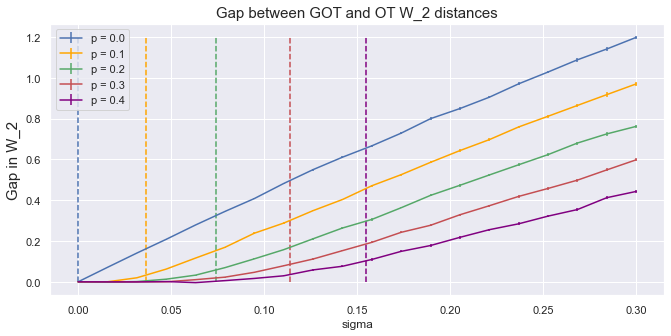}
        \caption{Rate of $W_2(\mu, \nu) - W_2(\mu\ast\sN_\sigma, \nu\ast\sN_\sigma)$ in the vanishing $\sigma$ regime.}
        \label{fig:rate}
    \end{minipage}
\end{figure}

By Theorem~\ref{thm:main-exp} and Theorem~\ref{thm:main-linear}, we expect $W_2^2(\mu_k, \nu) - W_2^2(\mu_k\ast\sN_\sigma, \nu\ast\sN_\sigma)$ to be of scale $e^{-c/\sigma^2}$ for $\sigma\in (0, R(\Gamma_k))$, and $W_2^2(\mu_k, \nu) - W_2^2(\mu_k\ast\sN_\sigma, \nu\ast\sN_\sigma) \gtrsim \sigma$ for $\sigma \ge R(\Gamma_k)$. The phase transition happening at the vertical dashed line (here $R(\Gamma_k)$ is evaluated per Proposition~\ref{prop:lb-R}) is indeed observed in the above experiment.

\section{Auxiliary proofs}
\label{sec:aux-proofs}
In this section, we provide proofs for Proposition~\ref{prop:equiv}, Proposition~\ref{prop:equiv-scm}, Proposition~\ref{prop:lb-eps-1}, Proposition~\ref{prop:lb-R}, Proposition~\ref{prop:max-local-improve}, Lemma~\ref{lemma:local-no-change} and Lemma~\ref{lemma:linear-loss-split}.
\begin{proof}[Proof of Proposition~\ref{prop:equiv}]
(i) to (ii). The idea is borrowed from~\cite{rocka1,rocka2,rochet}. Suppose $\Gamma$ is $f$-strongly cyclically monotone for a positive residual function $f$.
For $i\in [k]$, denote
\begin{equation*}
    v_i :=  \inf_{\substack{\theta(1) = 1, \theta(n+1) = i, \\\theta(2),\dots,\theta(n) \in [k],\\ \theta(s) \neq \theta(s+1)}} 
    \left(\sum_{s = 1}^n \la x_{\theta(s)}, y_{\theta(s)} - y_{\theta(s+1)}\ra 
     - \sum_{s = 1}^n f(\theta(s), \theta(s+1)) \right)
\end{equation*}
By the $f$-strong cyclical monotonicity, we have $v_1\ge 0$. Furthermore, for $i > 1$ and any sequence $\{\theta(s)\}$ with $\theta(1) = 1$, $\theta(n+1) = i$ and $\theta(s)\neq\theta(s+1)$, there holds
\begin{align*}
   \sum_{s = 1}^n \la x_{\theta(s)}, y_{\theta(s)} - y_{\theta(s+1)}\ra  + \la x_i, y_i - y_1\ra 
  \ge\sum_{s = 1}^n f(\theta(s), \theta(s+1)) +  f(i, 1)  
\end{align*}
and it follows that
\[v_i \ge  f(i, 1) - \la x_i, y_i-y_1\ra > -\infty.\]
For any $j\neq i$ and any fixed $\epsilon > 0$, there exists a sequence $\{\theta(s)\}$ with $\theta(1) = 1$, $\theta(n+1) = i$ and $\theta(s)\neq\theta(s+1)$, such that
\begin{equation}\label{ineq-vi}
    \sum_{s = 1}^n \la x_{\theta(s)}, y_{\theta(s)} - y_{\theta(s+1)}\ra - \sum_{s = 1}^n f(\theta(s), \theta(s+1)) \le v_i+\epsilon.
\end{equation}
Consider the same $\{\theta(s)\}$ with one more term $\theta(n+2) := j$. By definition of $v_j$ we have
\begin{equation}\label{ineq-vj}
    v_j \le  \sum_{s = 1}^n \la x_{\theta(s)}, y_{\theta(s)} - y_{\theta(s+1)}\ra + \la x_i, y_i-y_j \ra 
    - \sum_{s = 1}^{n+1} f(\theta(s), \theta(s+1))
\end{equation}
Comparing~\eqref{ineq-vi} and~\eqref{ineq-vj} we get
\begin{equation}\label{ineq_vivj}
    v_j \le v_i + \la x_i, y_i-y_j\ra - f(i, j) +\epsilon
\end{equation}
We set $\vp(x_i) = -v_i$. Letting $\epsilon\downarrow 0$ in~\eqref{ineq_vivj} yields
\begin{equation*}
    \la x_i, y_i - y_j \ra \ge \vp(x_i) - \vp(x_j) + f(i, j).
\end{equation*}
Hence $\Gamma$ is $f$-strongly implementable.\\
\noindent
(ii) to (iii). We prove by contradiction. Suppose $\Gamma$ is not the unique optimal transport plan; this means either $\Gamma$ is not optimal or there exists a different coupling $\Gamma'$ with the same cost. Either case, there exists a sequence $\{\theta(s)\}_{s = 1}^n$ such that
\[\sum_{s = 1}^n \|x_{\theta(s)}- y_{\theta(s)}\|^2 \ge \sum_{s = 1}^n \|x_{\theta(s)}- y_{\theta(s+1)}\|^2\]
Summing over $s$, we get
\begin{equation*}
    \begin{aligned}
     \sum_{s = 1}^n f(\theta(s), \theta(s+1)) 
     \le & \sum_{s = 1}^n \la x_{\theta(s)}, y_{\theta(s)} - y_{\theta(s+1)} \ra\\
     =&\ \frac{1}{2}\left(\sum_{s = 1}^n \|x_{\theta(s)}- y_{\theta(s+1)}\|^2 - \sum_{s = 1}^n \|x_{\theta(s)} - y_{\theta(s)}\|^2\right)\\
     \le&\ 0,
    \end{aligned}
\end{equation*}
a contradiction.\\
\noindent
(iii) to (i). Suppose $\Gamma$ is the unique optimal transport plan from $\{x_i\}$ to $\{y_i\}$. Denote $c_0$ the transport cost of $\Gamma$. For any other transport plan in the form of a bijection between $\{x_i\}$ and $\{y_i\}$, denote $c_1$ the minimum among their costs, then $c_1 > c_0$. Choose a small enough $\lambda > 0$, such that for any choice of $\sigma(1), \sigma(2), \dots, \sigma(n)\in [k]$ with no duplicates, there holds
\[\frac{\lambda}{2}\sum_{i = 1}^n \|y_{\sigma(i)} - y_{\sigma(i+1)}\|^2 \le c_1 - c_0.\]
Now for $f(i, j) = \frac{\lambda}{2}\|y_i - y_j\|^2$ we have
\begin{align*}
   \sum_{i = 1}^n \|x_{\sigma(i)} - y_{\sigma(i+1)}\|^2 - \sum_{i = 1}^n \|x_{\sigma(i)}- y_{\sigma(i)}\|^2
\ge c_1 - c_0 \ge \sum_{i = 1}^n f(\sigma(i), \sigma(i+1)). 
\end{align*}
If there are duplicates in $(\sigma(1), \sigma(2), \dots, \sigma(n))$, we break the loop $\sigma(1)\rightarrow\sigma(2)\rightarrow\dots\rightarrow\sigma(n)\rightarrow\sigma(1)$ into separate loops without duplicates, apply the above inequality to each loop and sum them up. We conclude by definition that $\Gamma$ is $f$-strongly cyclically monotone.
\end{proof}

\begin{proof}[Proof of Proposition~\ref{prop:equiv-scm}]
Suppose $\Gamma$ is $f$-strongly cyclically monotone for some positive residual $f$. Denote 
\[M := \max\left\{\max_i \|x_i\|, \max_i \|y_i\|\right\}.\]
We will show that $\Gamma$ is $\epsilon$-robust for any $\epsilon>0$ satisfying
\[4M\epsilon < \min_{i\neq j}f(i, j).\]
In fact, for any distinct $\sigma(1), \sigma(2), \dots, \sigma(n) \in [k]$, by the definition of $f$-strong cyclical monotonicity, 
\[\sum_{i = 1}^n \la x_{\sigma(i)}, y_{\sigma(i)} - y_{\sigma(i+1)} \ra \ge \sum_{i = 1}^n f(\sigma(i), \sigma(i+1))\]
Thus for any choice of $\alpha_{\sigma(1)}, \dots, \alpha_{\sigma(n)}$ such that $\max \|\alpha_{\sigma(i)}\| \le \epsilon$, we have
\begin{equation*}
    \begin{aligned}
    & \frac{1}{2}\sum_{i = 1}^n \|(x_{\sigma(i)}+\alpha_{\sigma(i)})
     - (y_{\sigma(i+1)} + \alpha_{\sigma(i+1)})\|^2 - \frac{1}{2}\sum_{i = 1}^n \|x_{\sigma(i)} - y_{\sigma(i)}\|^2 \\
     =& \sum_{i = 1}^n \la x_{\sigma(i)}, y_{\sigma(i)} - y_{\sigma(i+1)} \ra 
     + \sum_{i = 1}^n \la \alpha_{\sigma(i)}, x_{\sigma(i)} - x_{\sigma(i-1)} + y_{\sigma(i)} - y_{\sigma(i+1)} \ra +\frac{1}{2}\sum_{i = 1}^n \|\alpha_{\sigma(i)} - \alpha_{\sigma(i+1)}\|^2\\
     \ge & \sum_{i = 1}^n f(\sigma(i), \sigma(i+1)) - 4nM\epsilon\\
     > &\  0.
    \end{aligned}
\end{equation*}
Hence $R(\Gamma) > 0$. 

On the other hand, given $R(\Gamma) > 0$, we show that $\Gamma$ is the unique optimal transport plan from $\{x_i\}$ to $\{y_i\}$. We prove by contradiction. If $\Gamma$ is not unique, then there exists distinct $\sigma(1), \dots, \sigma(n) \in [k]$ such that
\begin{equation}
    \label{cost-equal}
    \sum_{i = 1}^n \|x_{\sigma(i)} - y_{\sigma(i)}\|^2 = \sum_{i = 1}^n \|x_{\sigma(i)} - y_{\sigma(i+1)}\|^2.
\end{equation}
Since $R(\Gamma) > 0$, for $\epsilon_0 = R(\Gamma)/2$ and any choice of $\sigma(1), \dots, \sigma(n)$ with $\|\sigma(i)\|\le \epsilon_0$, we have
\begin{equation*}
\sum_{i = 1}^n \|x_{\sigma(i)} - y_{\sigma(i)}\|^2 
     \le \sum_{i = 1}^n \|(x_{\sigma(i)}+\alpha_{\sigma(i)})
     - (y_{\sigma(i+1)} + \alpha_{\sigma(i+1)})\|^2.
\end{equation*}
Specifically, for any $j\in [n]$, letting $\sigma(i) = 0$ for all $i\neq j$ in the above equation gives
\[2\la \alpha_{\sigma(j)}, x_{\sigma(j)} - y_{\sigma(j+1)} \ra  \le \|\alpha_{\sigma(j)}\|^2\]
for any $\alpha_{\sigma(j)}\in \RR^d$ with $\|\alpha_{\sigma(j)}\| \le \epsilon_0$. Therefore we must have
\[x_{\sigma(j)} = y_{\sigma(j+1)}, \quad \forall\ j\in [k].\]
Using~\eqref{cost-equal}, we also know that
\[x_{\sigma(j)} = y_{\sigma(j)}, \quad \forall\ j\in [k],\]
which violates the assumption that $\{y_i\}$ are distinct points in $\RR^d$. Thus we conclude that $\Gamma$ is unique; hence it is also strongly cyclically monotone due to Proposition~\ref{prop:equiv}.
\end{proof}

\begin{proof}[Proof of Proposition~\ref{prop:lb-eps-1}]
We only need to show that, for an $\epsilon$ satisfying~\eqref{eq:eps-up-bound}, and any choice of $\sigma(1), \sigma(2), \dots, \sigma(n) \in [k]$, and $\alpha(1), \dots, \alpha(n)$ with $\|\alpha(i)\| \le \epsilon$, there holds 
\begin{equation}
\label{ineq:cm-prop-0}
    \sum_i \|x_{\sigma(i)} - y_{\sigma(i)}\|^2  \le \sum_i \|(x_{\sigma(i)}+\alpha_{\sigma(i)}) - (y_{\sigma(i+1)}
     + \alpha_{\sigma(i+1)})\|^2.
\end{equation}
In fact, \eqref{ineq:cm-prop-0} is equivalent to
\begin{equation}
\label{cond:eps-robust-simple}
2\sum_i \la \alpha_{\sigma(i)}, y_{\sigma(i+1)} - y_{\sigma(i)} + x_{\sigma(i-1)} - x_{\sigma(i)} \ra 
\le 2\sum_i \la x_{\sigma(i)}, y_{\sigma(i)} - y_{\sigma(i+1)} \ra + \sum_i \|\alpha_{\sigma(i)}-\alpha_{\sigma(i+1)}\|^2
\end{equation}
Since $\|\alpha(i)\| \le \epsilon$ for all $i$, we have
\begin{equation*}
\label{calc:robust-lhs}
    \begin{aligned}
        & 2\sum_i \la \alpha_{\sigma(i)}, y_{\sigma(i+1)} - y_{\sigma(i)} + x_{\sigma(i-1)} - x_{\sigma(i)} \ra \\
        &\le 2\sum_i \epsilon \cdot \left(\|y_{\sigma(i+1)} - y_{\sigma(i)}\| + \|x_{\sigma(i+1)} - x_{\sigma(i)}\|\right)\\
        &\le \sum_i f(\sigma(i), \sigma(i+1))
    \end{aligned}
\end{equation*}
\noindent
where we used the choice of $\epsilon$ in the last inequality. In the meantime, strong implementability gives
\begin{equation*}
2\sum_i \la x_{\sigma(i)}, y_{\sigma(i)} - y_{\sigma(i+1)} \ra + \sum_i \|\alpha_{\sigma(i)}-\alpha_{\sigma(i+1)}\|^2 
    \ge \sum_{i} f(\sigma(i), \sigma(i+1)).
\end{equation*}
Therefore~\eqref{cond:eps-robust-simple} holds, which completes the proof.
\end{proof}

\begin{proof}[Proof of Proposition~\ref{prop:lb-R}]
Following the proof of Proposition~\ref{prop:lb-eps-1}, we only need to show that, for the residual $f(i, j)$ defined in Theorem~\ref{thm:convex-equiv}, there holds
\begin{equation}
\label{cond:eps-robust-simple-coro}
    2\sum_i \epsilon \cdot \left(\|y_{\sigma(i+1)} - y_{\sigma(i)}\| + \|x_{\sigma(i+1)} - x_{\sigma(i)}\|\right) \le \sum_i f(\sigma(i), \sigma(i+1)).
\end{equation}
By the choice of $\epsilon$, we have
\begin{equation*}
\begin{aligned}
& 2\sum_i \epsilon \cdot \left(\|y_{\sigma(i+1)} - y_{\sigma(i)}\| + \|x_{\sigma(i+1)} - x_{\sigma(i)}\|\right)\\
        & \le \sum_i \max\left\{\frac{1}{\beta}\|x_{\sigma(i+1)} - x_{\sigma(i)}\|^2 , \alpha \|y_{\sigma(i+1)} - y_{\sigma(i)}\|^2\right\}.
        \end{aligned}
\end{equation*}
\noindent
Meanwhile,
\begin{equation*}
\label{calc:robust-rhs}
    \begin{aligned}
    &\sum_i f(\sigma(i), \sigma(i+1)) \\
    &= \frac{1}{\beta - \alpha}\sum_i\left(\|x_{\sigma(i)} - x_{\sigma(i+1)}\|^2
         + \alpha\beta \|y_{\sigma(i)} - y_{\sigma(i+1)}\|^2 - 2\alpha \la y_{\sigma(i)} - y_{\sigma(i+1)}, x_{\sigma(i)} - x_{\sigma(i+1)}\ra\right)\\
         &\ge \frac{1}{\beta - \alpha}\sum_i\left(\|x_{\sigma(i)} - x_{\sigma(i+1)}\|^2
         + \alpha\beta \|y_{\sigma(i)} - y_{\sigma(i+1)}\|^2 - \alpha \left(\lambda\|x_{\sigma(i)} - x_{\sigma(i+1)}\|^2
         + \frac{1}{\lambda}\|y_{\sigma(i)} - y_{\sigma(i+1)}\|^2 \right)\right).
    \end{aligned}
\end{equation*}
The last inequality holds for any $\lambda > 0$ by the Cauchy-Schwarz inequality. Choosing $\lambda = 1/\beta$ and $\lambda = 1/\alpha$ yields
\begin{equation*}
\sum_i f(\sigma(i), \sigma(i+1))
    \ge \max\left\{\frac{1}{\beta}\|x_{\sigma(i+1)} - x_{\sigma(i)}\|^2, \alpha \|y_{\sigma(i+1)} - y_{\sigma(i)}\|^2\right\}.
\end{equation*}

Therefore~\eqref{cond:eps-robust-simple-coro} holds, which completes the proof.
\end{proof}

\begin{proof}[Proof of Proposition~\ref{prop:max-local-improve}]
For $M>0$, denote
\begin{equation*}
    g(m)  := \sup\left\{ \sum_{i = 1}^n \|x_{\sigma(i)} - y_{\sigma(i)}\|^2  - \sum_{i = 1}^n \|(x_{\sigma(i)}+\alpha_{\sigma(i)}) - (y_{\sigma(i+1)}+\alpha_{\sigma(i+1)})\|^2 :
     \ \max_i\|\alpha_{\sigma(i)}\| = m\right\},
\end{equation*}
then $G(M) = \sup\{g(m): m\in [0, M]\}$. We first prove that $g(m)$ is concave in $m$. In fact, denote the set
\begin{equation*}
    \sI =   \left\{  (\sigma(1), \dots, \sigma(n), \alpha_{\sigma(1)}, \dots, \alpha_{\sigma(n)}):\
     \sigma(i)\in [k], \ \sigma(i)\neq\sigma(j),\ \max_i \|\alpha_{\sigma(i)}\| = 1 \right\}.
\end{equation*}
By definition, 
\begin{equation*}
    \begin{aligned}
     g(m) & = \sup \left\{ \sum_{i = 1}^n \|x_{\sigma(i)} - y_{\sigma(i)}\|^2 
      - \sum_{i = 1}^n \|(x_{\sigma(i)}+m\alpha_{\sigma(i)}) - (y_{\sigma(i+1)}+m\alpha_{\sigma(i+1)})\|^2 :\right.\\
     &\left. \quad (\sigma(1), \dots, \sigma(n), \alpha_{\sigma(1)}, \dots, \alpha_{\sigma(n)}) \in \sI\right\}
    \end{aligned}
\end{equation*}
Note that, for every choice of $(\sigma(1), \dots, \sigma(n))$ and  $\alpha_{\sigma(1)}, \dots, \alpha_{\sigma(n)}) \in \sI$, 
\begin{align*}
    \sum_{i = 1}^n \|x_{\sigma(i)} - y_{\sigma(i)}\|^2
     - \sum_{i = 1}^n \|(x_{\sigma(i)}+m\alpha_{\sigma(i)}) - (y_{\sigma(i+1)}+m\alpha_{\sigma(i+1)})\|^2
\end{align*}
is a concave function in $m$. Therefore, $g(m)$ is concave in $m$, and $G(M)$ is also concave in $M$.
\end{proof}

\begin{proof}[Proof of Lemma~\ref{lemma:linear-loss-split}]
First suppose that $x, y_1, y_2$ are not on the same line with $y_1$ between $x$ and $y_2$ or $y_2$ between $x$ and $y_1$. Let $\Delta$ be the bisecting hyperplane of $\angle y_1 x y_2$, namely
\[\Delta = \left\{z\in\RR^d\ :\ \frac{\la z - x, y_1 - x\ra}{|y_1 - x|} = \frac{\la z - x, y_2 - x \ra }{|y_2 - x|}\right\},\]
and define its unit normal vector $\mathbf{m}$ such that $\la \mathbf{m}, y_1 - x \ra > 0$. We adopt the decomposition
\begin{equation}
    \begin{aligned}
     \mu_{+} &:= \sN(x, \sigma^2)\ | \ \la z-x, \mathbf{m} \ra > 0, \\
     \mu_{-} &:= \sN(x, \sigma^2)\ | \ \la z-x, \mathbf{m} \ra < 0,
    \end{aligned}
\end{equation}
and
\begin{equation}
    \begin{aligned}
    \nu_{1+} &:= \sN(y_1, \sigma^2)\ | \ \la z-y_1, \mathbf{m} \ra > 0, \\
    \nu_{1-} & := \sN(y_1, \sigma^2)\ | \ \la z-y_1, \mathbf{m} \ra < 0, \\
    \nu_{2+} &:= \sN(y_2, \sigma^2)\ | \ \la z-y_2, \mathbf{m} \ra > 0, \\
    \nu_{2-} & := \sN(y_2, \sigma^2)\ | \ \la z-y_2, \mathbf{m} \ra < 0.
    \end{aligned}
\end{equation}
Note that all the six sub-probability measures above have mass $1/2$. By the definition of $W_2$, we have
\begin{equation}
\label{eq:w2-split}
    W_2^2(\mu_0\ast\sN_\sigma, \nu_0\ast\sN_\sigma) \le 
    \frac{1}{2}\left(W_2^2(\mu_+, \nu_{1+}) + W_2^2(\mu_+, \nu_{1-}) + W_2^2(\mu_-, \nu_{2+}) + W_2^2(\mu_-, \nu_{2-})\right).
\end{equation}
It is obvious that 
\[W_2^2(\mu_+, \nu_{1+}) = \frac{1}{2}\|x - y_1\|^2, \quad W_2^2(\mu_-, \nu_{2-}) = \frac{1}{2}\|x-y_2\|^2.\]
For $W_2^2(\mu_+, \nu_{1-})$, consider the map
\[T_{\#}(x + t)\ =\ y_1 - t, \quad t \sim \sN(0, \sigma^2 I) \]
we have
\begin{equation*}
\begin{aligned}
    W_2^2(\mu_+, \nu_{1-}) &\le \EE_{u\sim \mu_+} \|u - T_{\#}u\|^2\\
    &= \EE_{u\sim\mu_+} \|u - (y_1-u+x)\|^2\\
    &= \frac{1}{2}\|x-y_1\|^2 - 4 \EE_{u\sim\mu_+} \la y_1 - x, u - x\ra + 4\EE_{u\sim\mu_+} \|u-x\|^2\\
    &= \frac{1}{2}\|x-y_1\|^2 - 4c_1\sigma \la \mathbf{m}, y_1 - x\ra + 4c_2\sigma^2,
    \end{aligned}
\end{equation*}
where $c_1$ and $c_2$ are absolute positive constants. Similarly, 
\begin{equation*}
    W_2^2(\mu_-, \nu_{2+}) \le \frac{1}{2}\|x-y_2\|^2 - 4c_1\sigma\la \mathbf{m}, x - y_2\ra + 4c_2\sigma^2.
\end{equation*}
Plugging into~\eqref{eq:w2-split} we get
\[ W_2^2(\mu_0\ast\sN_\sigma, \nu_0\ast\sN_\sigma) \le W_2^2(\mu_0, \nu_0) - 4c_1\sigma \la \mathbf{m}, y_1 - y_2\ra + 8c_2\sigma^2,\]
hence $W_2^2(\mu_0, \nu_0) - W_2^2(\mu_0\ast\sN_\sigma, \nu_0\ast\sN_\sigma) \gtrsim \sigma$ for small $\sigma$, since $\la \mathbf{m}, y_1 - y_2\ra > 0$.\\

Finally, we consider the special case where $x, y_1, y_2$ are on the same line and $y_1$ is between $x$ and $y_2$. We choose $\mathbf{m}$ the unit vector along the direction $x - y_1$, and the same line of proof yields the conclusion.
\end{proof}

\begin{proof}[Proof of Lemma~\ref{lemma:local-no-change}]
We naturally split the source measure into $k$ parts:
\[\mu \ast Q = \sum_{i = 1}^k\left(\alpha_i\delta(x_i)\ast Q\right)\]
Consider a map $T$ which, for each $i\in [k]$, is defined by
\[T(x) = x+y_i - x_i \quad \quad \forall x \in B(x_i, \sigma_*)\,.\]
We can obtain a transport plan between $\mu * Q$ and $\nu * Q$ by considering the distribution of a pair of random variables $(X, T(X))$ for $X \sim \mu * Q$.
The support of this plan lies in the set $\bigcup_{i=1}^k \bigcup_{\alpha \in B(0, \sigma_*)} (x_i + \alpha, y_i + \alpha)$.
By the definition of $R(\Gamma)$, this set is cyclically monotone, so this coupling is optimal for $\mu * Q$ and $\nu * Q$ by Theorem~\ref{monotone_opt}.
Therefore
\begin{equation*}
\begin{aligned}
W^2_2(\mu\ast Q, \nu\ast Q) & = \int \|x - T(x)\|^2 d(\mu * Q)(x) \\
& = \sum_{i=1}^k \alpha_i \|y_i - x_i\|^2 = W_2^2(\mu, \nu)\,,
\end{aligned}
\end{equation*}
as claimed.
\end{proof}
\begin{proof}[Proof of Theorem~\ref{thm:main-exp}]
Define the truncated smoothing kernel
\[\tilde{\sN}_{\sigma} := p \sN(0,\sigma^2 I | \|X\|< \epsilon_\ast) + (1-p)\delta(0) \]
where
\[p = \PP\left[\|\sN(0, \sigma^2 I)\| < \epsilon_\ast\right].\]
Since $\tilde{\sN}_\sigma$ is supported on $B(0, \epsilon_\ast)$, by Lemma~\ref{lemma:local-no-change}, we know
\[W_2(\mu\ast \tilde{\sN}_{\sigma}, \nu\ast\tilde{\sN}_{\sigma})  = W_2(\mu, \nu).\]
Therefore,
\begin{align*}
    &\ |W_2(\mu\ast \sN_\sigma, \nu\ast\sN_\sigma)-W_2(\mu, \nu)|^2 \\
    =&\ |W_2(\mu\ast \sN_\sigma, \nu\ast\sN_\sigma)-W_2(\mu\ast \tilde{\sN}_\sigma, \nu\ast\tilde{\sN}_\sigma)|^2\\
    \le&\ (W_2(\mu\ast \sN_\sigma, \mu\ast \tilde{\sN}_\sigma) + W_2(\nu\ast \sN_\sigma, \nu\ast \tilde{\sN}_\sigma))^2\\
    \lesssim &\ \EE_{z\sim \sN(0, \sigma^2 I)} \left[\|z\|^2 \one_{\|z\|\ge \sigma_\ast}\right]\\
    =&\ \sigma^2\ \EE_{z\sim \sN(0,I)} \left[\|z\|^2 \one_{\|z\|\ge \sigma_\ast/\sigma}\right]\\
    \lesssim &\ \sigma \sigma_\ast e^{-\sigma_\ast^2 / 2\sigma^2}.
\end{align*}
Taking square root on both sides yields the result.
\end{proof}


\bibliographystyle{alpha}
\bibliography{main}

\newcommand{\etalchar}[1]{$^{#1}$}
\begin{thebibliography}{GGNWP20}

\bibitem[ACB17]{ArjChiBot17}
Martin Arjovsky, Soumith Chintala, and L{\'e}on Bottou.
\newblock Wasserstein {GAN}.
\newblock {\em arXiv preprint arXiv:1701.07875}, 2017.

\bibitem[AWR17]{awr-sinkhorn}
Jason Altschuler, Jonathan Weed, and Philippe Rigollet.
\newblock Near-linear time approximation algorithms for optimal transport via
  sinkhorn iteration.
\newblock In {\em Proceedings of the 31st International Conference on Neural
  Information Processing Systems}, pages 1961--1971, 2017.

\bibitem[BLG14]{blg-empirical}
Emmanuel Boissard and Thibaut Le~Gouic.
\newblock On the mean speed of convergence of empirical and occupation measures
  in wasserstein distance.
\newblock In {\em Annales de l'IHP Probabilit{\'e}s et statistiques},
  volume~50, pages 539--563, 2014.

\bibitem[Bre87]{Bre87}
Yann Brenier.
\newblock D\'{e}composition polaire et r\'{e}arrangement monotone des champs de
  vecteurs.
\newblock {\em C. R. Acad. Sci. Paris S\'{e}r. I Math.}, 305(19):805--808,
  1987.

\bibitem[CFT14]{CouFlaTui14}
Nicolas Courty, R{\'{e}}mi Flamary, and Devis Tuia.
\newblock Domain adaptation with regularized optimal transport.
\newblock In {\em ECML PKDD}, pages 274--289, 2014.

\bibitem[CFTR17]{CouFlaTui17}
Nicolas Courty, R{\'{e}}mi Flamary, Devis Tuia, and Alain Rakotomamonjy.
\newblock Optimal transport for domain adaptation.
\newblock {\em {IEEE} Trans. Pattern Anal. Mach. Intell.}, 39(9):1853--1865,
  2017.

\bibitem[CGT19]{DBLP:journals/access/ChenGT19}
Yongxin Chen, Tryphon~T. Georgiou, and Allen~R. Tannenbaum.
\newblock Optimal transport for gaussian mixture models.
\newblock {\em {IEEE} Access}, 7:6269--6278, 2019.

\bibitem[CNW20]{CN-asymp-large-noise}
Hong-Bin Chen and Jonathan Niles-Weed.
\newblock Asymptotics of smoothed wasserstein distances.
\newblock {\em arXiv preprint arXiv:2005.00738}, 2020.

\bibitem[Cut13]{cuturi-sinkhorn}
Marco Cuturi.
\newblock Sinkhorn distances: lightspeed computation of optimal transport.
\newblock In {\em NIPS}, volume~2, page~4, 2013.

\bibitem[DD20]{doi:10.1137/19M1301047}
Julie Delon and Agnès Desolneux.
\newblock A wasserstein-type distance in the space of gaussian mixture models.
\newblock {\em SIAM Journal on Imaging Sciences}, 13(2):936--970, 2020.

\bibitem[Dud69]{dudley-convergence}
Richard~Mansfield Dudley.
\newblock The speed of mean glivenko-cantelli convergence.
\newblock {\em The Annals of Mathematical Statistics}, 40(1):40--50, 1969.

\bibitem[DY95]{dobric-asymptotics}
V~Dobri{\'c} and Joseph~E Yukich.
\newblock Asymptotics for transportation cost in high dimensions.
\newblock {\em Journal of Theoretical Probability}, 8(1):97--118, 1995.

\bibitem[FCVP17]{FeyChaVia17}
Jean Feydy, Benjamin Charlier, Fran{\c{c}}ois{-}Xavier Vialard, and Gabriel
  Peyr{\'{e}}.
\newblock Optimal transport for diffeomorphic registration.
\newblock In {\em Medical Image Computing and Computer Assisted Intervention -
  {MICCAI} 2017 - 20th International Conference, Quebec City, QC, Canada,
  September 11-13, 2017, Proceedings, Part {I}}, pages 291--299, 2017.

\bibitem[FG15]{FG-rate}
Nicolas Fournier and Arnaud Guillin.
\newblock On the rate of convergence in wasserstein distance of the empirical
  measure.
\newblock {\em Probability Theory and Related Fields}, 162(3):707--738, 2015.

\bibitem[GBG{\etalchar{+}}18]{goldfeld-estim-NN}
Ziv Goldfeld, Ewout van~den Berg, Kristjan Greenewald, Igor Melnyk, Nam Nguyen,
  Brian Kingsbury, and Yury Polyanskiy.
\newblock Estimating information flow in deep neural networks.
\newblock {\em arXiv preprint arXiv:1810.05728}, 2018.

\bibitem[GCB{\etalchar{+}}19]{genevay-complexity}
Aude Genevay, L{\'e}naic Chizat, Francis Bach, Marco Cuturi, and Gabriel
  Peyr{\'e}.
\newblock Sample complexity of sinkhorn divergences.
\newblock In {\em The 22nd International Conference on Artificial Intelligence
  and Statistics}, pages 1574--1583. PMLR, 2019.

\bibitem[GG20]{GG-gaussian}
Ziv Goldfeld and Kristjan Greenewald.
\newblock Gaussian-smoothed optimal transport: Metric structure and statistical
  efficiency.
\newblock In {\em International Conference on Artificial Intelligence and
  Statistics}, pages 3327--3337. PMLR, 2020.

\bibitem[GGK20]{GGK-generative}
Ziv Goldfeld, Kristjan Greenewald, and Kengo Kato.
\newblock Asymptotic guarantees for generative modeling based on the smooth
  wasserstein distance.
\newblock {\em arXiv preprint arXiv:2002.01012}, 2020.

\bibitem[GGNWP20]{GGNP-empirical}
Ziv Goldfeld, Kristjan Greenewald, Jonathan Niles-Weed, and Yury Polyanskiy.
\newblock Convergence of smoothed empirical measures with applications to
  entropy estimation.
\newblock {\em IEEE Transactions on Information Theory}, 66(7):4368--4391,
  2020.

\bibitem[GKNR22]{GolKatNie22}
Ziv Goldfeld, Kengo Kato, Sloan Nietert, and Gabriel Rioux.
\newblock Limit distribution theory for smooth $p$-wasserstein distances, 2022.

\bibitem[GPC18]{GenPeyCut18}
Aude Genevay, Gabriel Peyr{\'{e}}, and Marco Cuturi.
\newblock Learning generative models with sinkhorn divergences.
\newblock In {\em International Conference on Artificial Intelligence and
  Statistics, {AISTATS} 2018, 9-11 April 2018, Playa Blanca, Lanzarote, Canary
  Islands, Spain}, pages 1608--1617, 2018.

\bibitem[MNW19]{MN-entropic}
Gonzalo Mena and Jonathan Niles-Weed.
\newblock Statistical bounds for entropic optimal transport: Sample complexity
  and the central limit theorem.
\newblock {\em Advances in Neural Information Processing Systems}, 32, 2019.

\bibitem[MNW21]{ManNil21}
Tudor Manole and Jonathan Niles-Weed.
\newblock Sharp convergence rates for empirical optimal transport with smooth
  costs, 2021.

\bibitem[NWB19]{NB-minimax-smooth}
Jonathan Niles-Weed and Quentin Berthet.
\newblock Minimax estimation of smooth densities in wasserstein distance.
\newblock {\em arXiv e-prints}, pages arXiv--1902, 2019.

\bibitem[NWR19]{NR-stm}
Jonathan Niles-Weed and Philippe Rigollet.
\newblock Estimation of wasserstein distances in the spiked transport model.
\newblock {\em arXiv preprint arXiv:1909.07513}, 2019.

\bibitem[PdC20]{paty2020regularity}
Fran{\c{c}}ois-Pierre Paty, Alexandre d’Aspremont, and Marco Cuturi.
\newblock Regularity as regularization: Smooth and strongly convex brenier
  potentials in optimal transport.
\newblock In {\em International Conference on Artificial Intelligence and
  Statistics}, pages 1222--1232. PMLR, 2020.

\bibitem[PKD07]{PitKokDah07}
Fran{\c{c}}ois Piti{\'{e}}, Anil~C. Kokaram, and Rozenn Dahyot.
\newblock Automated colour grading using colour distribution transfer.
\newblock {\em Computer Vision and Image Understanding}, 107(1-2):123--137,
  2007.

\bibitem[Roc66]{rocka1}
Ralph Rockafellar.
\newblock Characterization of the subdifferentials of convex functions.
\newblock {\em Pacific Journal of Mathematics}, 17(3):497--510, 1966.

\bibitem[Roc70]{rocka2}
Ralph Rockafellar.
\newblock On the maximal monotonicity of subdifferential mappings.
\newblock {\em Pacific Journal of Mathematics}, 33(1):209--216, 1970.

\bibitem[Roc87]{rochet}
Jean-Charles Rochet.
\newblock A necessary and sufficient condition for rationalizability in a
  quasi-linear context.
\newblock {\em Journal of mathematical Economics}, 16(2):191--200, 1987.

\bibitem[RTG00]{RubTomGui00}
Yossi Rubner, Carlo Tomasi, and Leonidas~J Guibas.
\newblock The earth mover's distance as a metric for image retrieval.
\newblock {\em International journal of computer vision}, 40(2):99--121, 2000.

\bibitem[RW18]{RW-entropic}
Philippe Rigollet and Jonathan Weed.
\newblock Entropic optimal transport is maximum-likelihood deconvolution.
\newblock {\em Comptes Rendus Mathematique}, 356(11-12):1228--1235, 2018.

\bibitem[SP18]{SP-minimax}
Shashank Singh and Barnab{\'a}s P{\'o}czos.
\newblock Minimax distribution estimation in wasserstein distance.
\newblock {\em arXiv preprint arXiv:1802.08855}, 2018.

\bibitem[SUL{\etalchar{+}}18]{SinUppLi18}
Shashank Singh, Ananya Uppal, Boyue Li, Chun-Liang Li, Manzil Zaheer, and
  Barnab{\'a}s P{\'o}czos.
\newblock Nonparametric density estimation under adversarial losses.
\newblock {\em arXiv preprint arXiv:1805.08836}, 2018.

\bibitem[THG17]{TayHenGli17}
Adrien~B. Taylor, Julien~M. Hendrickx, and Fran\c{c}ois Glineur.
\newblock Smooth strongly convex interpolation and exact worst-case performance
  of first-order methods.
\newblock {\em Math. Program.}, 161(1-2, Ser. A):307--345, 2017.

\bibitem[Vil08]{Vil08}
C{\'e}dric Villani.
\newblock {\em Optimal transport: old and new}, volume 338.
\newblock Springer Science \& Business Media, 2008.

\bibitem[VV21]{vacher2021convex}
Adrien Vacher and Fran{\c{c}}ois-Xavier Vialard.
\newblock Convex transport potential selection with semi-dual criterion.
\newblock {\em arXiv preprint arXiv:2112.07275}, 2021.

\bibitem[ZCR21]{ZCR-smoothed}
Yixing Zhang, Xiuyuan Cheng, and Galen Reeves.
\newblock Convergence of gaussian-smoothed optimal transport distance with
  sub-gamma distributions and dependent samples.
\newblock In {\em International Conference on Artificial Intelligence and
  Statistics}, pages 2422--2430. PMLR, 2021.

\end{thebibliography}

\end{document}